\documentclass[reqno, 10pt]{amsart}
\usepackage{amssymb,amsmath,amsfonts,amsthm}
\usepackage[mathscr]{euscript}
\usepackage{hyperref}
\usepackage{graphicx}
\usepackage[normalem]{ulem}

\usepackage{hyperref}
\RequirePackage[dvipsnames]{xcolor} 
\definecolor{halfgray}{gray}{0.55} 
\definecolor{webgreen}{rgb}{0,0.5,0}
\definecolor{webbrown}{rgb}{.6,0,0} \hypersetup{%
  colorlinks=true, linktocpage=true, pdfstartpage=3,
  pdfstartview=FitV,%
  breaklinks=true, pdfpagemode=UseNone, pageanchor=true,
  pdfpagemode=UseOutlines,%
  plainpages=false, bookmarksnumbered, bookmarksopen=true,
  bookmarksopenlevel=1,%
  hypertexnames=true,
  pdfhighlight=/O,
  urlcolor=webbrown, linkcolor=RoyalBlue,
  citecolor=webgreen, 
  pdftitle={},%
  pdfauthor={},%
  pdfsubject={2000 MAthematical Subject Classification: Primary:},%
  pdfkeywords={},%
  pdfcreator={pdfLaTeX},%
  pdfproducer={LaTeX with hyperref}%
}

\usepackage{enumitem}

\newtheorem{theorem}{Theorem}[section]
\newtheorem{lemma}[theorem]{Lemma}
\newtheorem{proposition}[theorem]{Proposition}
\newtheorem{corollary}[theorem]{Corollary}
\theoremstyle{definition}

\newtheorem{remark}[theorem]{Remark}

\def\Pd{\mathbb{P}^{d-1}}
\def\P{{\mathbb{P}}}
\def\real{\mathbb{R}}
\def\P{{\mathbb{P}}}
\def\exteriorj{\Lambda^{j}\left(\real^d\right)}
\def\grassj{\textrm{Grass}(j,d)}
\newcommand{\ip}[1]{{\left\langle \, #1 \, \right\rangle}}
\newcommand{\norm}[1]{{\left\lVert \, #1 \, \right\rVert}}
\def\ker{\operatorname{Ker}}
\def\d{\operatorname{dist}}

\keywords{Lyapunov exponents, Oseledets subspaces, continuity, linear cocycles}


\title[Lyapunov exponents versus Oseledets subspaces]{Continuity of Lyapunov exponents is equivalent to continuity of Oseledets subspaces}

\author{Lucas Backes}

\address{\noindent Departamento de Matem\'atica, Universidade Federal do Rio Grande do Sul, Av. Bento Gon\c{c}alves 9500, CEP 91509-900, Porto Alegre, RS, Brazil.
\newline e-mail: \rm
  \texttt{lhbackes@impa.br} }

\author{Mauricio Poletti}

\address{IMPA - Estrada D. Castorina 110, Jardim Bot\^anico, CEP 22460-320, Rio de Janeiro, RJ, Brazil.
\newline e-mail: \rm
  \texttt{mpoletti@impa.br}  }

\date{\today}

\begin{document}

\begin{abstract}
We prove that, for semi-invertible continuous cocycles, continuity of Lyapunov exponents is equivalent to continuity, in measure, of Oseledets subspaces.
\end{abstract}

\maketitle

\section{Introduction}

Consider an invertible ergodic measure preserving dynamical system $f:M \rightarrow M$ defined on a measure space $(M,\mathcal{A},\mu)$ and a measurable matrix-valued map $A:M\rightarrow M(d, \mathbb{R})$. The pair $(f,A)$ is called a \textit{semi-invertible linear cocycle} (or just \textit{linear cocycle} for short). Sometimes one calls linear cocycle (over $f$ generated by $A$), instead, the sequence $\lbrace A^n\rbrace _{n\in \mathbb{N}}$ defined by
\begin{equation*}\label{def:cocycles}
A^n(x)=
\left\{
	\begin{array}{ll}
		A(f^{n-1}(x))\ldots A(f(x))A(x)  & \mbox{if } n>0 \\
		Id & \mbox{if } n=0 \\
	\end{array}
\right.
\end{equation*}
for all $x\in M$. The word `semi-invertible' refers to the fact that the action of the underlying dynamical system $f$ is invertible while the action on the fibers given by $A$ may fail to be invertible.

Under certain integrability conditions, it was proved in \cite{FLQ10} that for $\mu$-almost every point $x\in M$ there exist numbers $\lambda _1>\ldots > \lambda _{l}\geq -\infty$, called \textit{Lyapunov exponents}, and a direct sum decomposition $\mathbb{R}^d=E^{1,A}_{x}\oplus \ldots \oplus E^{l,A}_{x}$ into vector subspaces which are called \textit{Oseledets subspaces} and depend measurable on $x$ such that, for every $1\leq i \leq l$,
\begin{itemize}
\item dim$(E^{i,A}_{x})$ is constant,
\item $A(x)E^{i,A}_{x}\subseteq E^{i,A}_{f(x)}$ with equality when $\lambda _i>-\infty$
\end{itemize}
and
\begin{itemize}
\item $\lambda _i =\lim _{n\rightarrow +\infty} \dfrac{1}{n}\log \parallel A^n(x)v\parallel$ 
for every non-zero $v\in E^{i,A}_{x}$.
\end{itemize}
This result extends a famous theorem due to Oseledets \cite{Ose68} known as the \textit{multiplicative ergodic theorem} which was originally stated in both, \textit{invertible} (both $f$ and the matrices are assumed to be invertible) and \textit{non-invertible} (neither $f$ nor the matrices are assumed to be invertible) settings (see also \cite{LLE}). While in the invertible case the conclusion is similar to the conclusion above (except that all Lyapunov exponents are finite), in the non-invertible case, instead of a direct sum decomposition into invariant vector subspaces, one only get an invariant \textit{filtration} (a sequence of nested subspaces) of $\real ^d$.

Lyapunov exponents are one of the most fundamental concepts in dynamical systems. For instance, the non-vanishing of the Lyapunov exponents of the derivative cocycle associated to a smooth dynamical system is the starting point for the whole branch of nonuniform hyperbolic theory (see \cite{BaP07}). As such important objects, it is natural that one tries to understand their behavior and describe its properties. 

One aspect that one could be interested in is, for instance, how do they vary when we perturb the cocycle $A$ while keeping the base dynamics $f$ fixed. It is well known that even in the invertible setting Lyapunov exponents may be highly discontinuous as functions of the cocycle \cite{Boc-un, Boc02}. Nevertheless, there are settings where one can get continuity \cite{BockerV, BBB, AvEV, MV15} and even real-analyticity \cite{Rue79a, Pe91}.

Similarly, one could be interested in understanding the continuity properties of the Oseledets subspaces (see \cite{ABF16, Bac, DK_b, DrF}). Since for each fixed $A$ and $1\leq i\leq l$ the map $x\to E^{i,A}_{x}$ is a measurable one, a natural notion of continuity to be considered would be continuity in measure of the map $A\to E^{i,A}_{x}$.

In the present work, rather than proving continuity of Lyapunov exponents or Oseledets subspaces themselves, we are interested in understanding in the semi-invertible setting how continuity of Lyapunov exponents relates to continuity, in measure, of the Oseledets subspaces. Indeed, as a consequence of our main result we get that these notions are actually equivalent. More precisely (see Section \ref{sec: definitions and statements} for precise definitions and statements),

\begin{theorem}
 $A$ is a continuity point for the Lyapunov exponents if and only if it is a continuity point for the Oseledets subspaces.
\end{theorem}

As a simple yet interesting application of our result combined with Theorem 5 of \cite{BcV05} we get that, in the invertible setting, if $A$ is a continuity point for the Oseledets subspaces then the Oseledets splitting at $x$ is either dominated or trivial at $\mu$-almost every $x\in M$. While the first option is quite expected since dominated splittings vary continuously, the second one is not so evident. Moreover, using results of \cite{BeS16} we get a similar conclusion in the semi-invertible setting when restricted to stochastic matrices.

The proof of our main result is based on an analysis of the dynamics `induced' by the cocycle on the projective space. Observe that, since the map $A$ takes values in $M(d,\real)$, its natural action on the projective space is not well defined. We overcome this issue introducing the notion of what we have called ``semi-projective cocycle". 

In the invertible setting and under random perturbations a similar stability result was gotten in \cite{Oc99}.

\section{Definitions and Statements}\label{sec: definitions and statements}

Let $(M,d)$ be a compact metric space, $\mu$ a measure defined on the Borel sets of $(M,d)$ and $f: M \to M $ a measure preserving homeomorphism. Assume also that $\mu$ is ergodic. Given a continuous map $A: M\to M(d,\mathbb{R})$ such that $\int \log ^{+}\norm{A(x)}d\mu (x)<\infty$, let us denote by 
\begin{displaymath}
\lambda _1 (A)>\lambda _2 (A) > \dots >\lambda _l (A)\geq -\infty
\end{displaymath}
the Lyapunov exponents of the cocycle $(f,A)$, by $d_i(A)$ the dimension of the Oseledets subspace associated with $\lambda _i(A)$ for $1\leq i\leq l$ and by
\begin{displaymath}
\gamma _1 (A)\geq \gamma _2 (A) \geq \dots \geq \gamma _d (A)
\end{displaymath}
the Lyapunov exponents of $(f,A)$ counted with multiplicities. Since our base dynamics is going to be fixed, we are going to refer to $\lambda _i(A)$ and $\gamma _i(A)$ simply as the Lyapunov exponents of $A$. 

Let $C^0(M)$ be the space of continuous maps $A: M\to M(d,\mathbb{R})$. We endow this space with the uniform topology which is generated by norm
\begin{displaymath}
\parallel A \parallel _{\infty}:= \sup _{x\in M} \parallel A(x)\parallel.
\end{displaymath}

We say that $A\in C^0(M)$ is a \textit{continuity point for the Lyapunov exponents} if for every sequence $\{A_k\}_k \subset C^0(M)$ converging to $A$ we have $\lim_{k\to \infty}\gamma_i(A_k)= \gamma _i(A)$ for every $1\leq i \leq d$. Observe that in this case for every $k$ sufficiently large we have
$$\gamma _1(A_k)\geq\gamma _{\tilde{d}_1}(A_k)>\gamma _{\tilde{d}_1+1}(A_k)\geq\gamma _{\tilde{d}_2}(A_k)>\ldots>\gamma _{\tilde{d}_{l-1}+1}(A_k)\geq\gamma _d(A_k)$$
where $\tilde{d}_i=\sum _{j=1}^{i}d_j(A)$ for every $1\leq i\leq l$. In particular, $A_k$ has at least $l$ different Lyapunov exponents and the sum of the dimensions of the Oseledets subspaces associated with $\gamma _{\tilde{d}_{j-1}+1}(A_k),\ldots ,\gamma _{\tilde{d}_j}(A_k)$ coincide with the dimension of $E^{j,A}_x$ for every $1\leq j\leq l$ where $\tilde{d}_0=0$. This motivates the following definition.

Given a sequence $\{A_k\}_k \subset C^0(M)$ converging to $A\in C^0(M)$ we say that the \textit{Oseledets subspaces of $A_k$ converge to those of $A$ with respect to the measure $\mu$} if for every $k$ sufficiently large there exists a direct sum decomposition $\mathbb{R}^d=F^{1,A_k}_{x}\oplus \ldots \oplus F^{l,A_k}_{x}$ into vector subspaces such that the following conditions are satisfied:
\begin{itemize}
\item[i)] $F^{i,A_k}_{x}=E^{j,A_k}_{x}\oplus E^{j+1,A_k}_{x}\oplus \ldots \oplus E^{j+t,A_k}_{x}$ for some $j\in \{1,\ldots ,l_k\}$ and $t\geq 0$;
\item[ii)] $\mbox{dim}(F^{i,A_k}_{x})= \mbox{dim}(E^{i,A}_{x})$ for every $i=1,\ldots ,l$;
\item[iii)] for every $\delta >0$ and $1\leq i\leq l$ we have 
$$\mu\left(\{ x\in M; \measuredangle(F^{i,A_k}_x,E^{i,A}_x )>\delta \} \right) \xrightarrow{k\to \infty} 0$$
\end{itemize}
where the \textit{angle} $\measuredangle{(E,F)}$ between two subspaces $E$ and $F$ of $\real ^d$ is defined as follows: given $w\in \real^d$ we define 
$$\d(w,E)=\inf_{v\in E} \norm{ w-v}.$$ 
It is easy to see that $\d(w,E)=\norm{w^{\perp}}$ where $w^{\perp}=w-\mbox{Proj}_{E}w$. More generally, we may consider the distance between $E$ and $F$ given by
\begin{equation}\label{def:distancia}
\d(E,F)=\sup_{v\in E,w\in F}\left\{ \d\left(\frac{v}{\norm{v}},F\right),\d\left(\frac{w}{\norm{w}},E\right)\right\}.
\end{equation}
Then, the angle between $E$ and $F$ is just $\measuredangle(E,F)=\sin^{-1}(\d(E,F))$. A cocycle $A$ is said to be a \textit{continuity point for the Oseledets decomposition with respect to the measure $\mu$} if the above requirements are satisfied for every sequence $\{A_k\}_k \subset C^0(M)$ converging to $A$.

Thus, our main result is the following.

\begin{theorem}\label{teo.principal}
Let $\{A_k\}_k \subset C^0(M)$ be a sequence converging to $A\in C^0(M)$. Then $\lim_{k\to \infty}\gamma_i(A_k)= \gamma _i(A)$ for every $1\leq i \leq d$ if and only if the Oseledets subspaces of $A_k$ converge to those of $A$ with respect to the measure $\mu$.
\end{theorem}

\begin{remark} Observe that requiring $\lim_{k\to \infty}\lambda_i(A_k)=\lambda_i(A)$ for every $i=1,\ldots ,l$ is not enough to guarantee that the Oseledets subspaces of $A_k$ converge to those of $A$ with respect to the measure $\mu$. For example, let $A:M\to \mbox{SL}(2,\real)$ be a discontinuity point for the Lyapunov exponents whose existence is guaranteed, for instance, by \cite{Boc-un, BockerV, But}. Thus, by the upper semi-continuity of the largest Lyapunov exponent there exists a sequence $\{A_k\}_k\subset C^0(M)$ converging to $A$ such that $\lim_{k\to \infty}\lambda_1(A_k)=a <\lambda_1(A)$. Now, considering $\hat{A}_k:M\to GL(6,\real)$ given by 
\begin{displaymath}
\hat{A}_k=\left( \begin{array}{ccccc}
A_k & \begin{matrix} 0\\0\end{matrix} & \begin{matrix} 0\\0\end{matrix}&\begin{matrix} 0\\0\end{matrix}&\begin{matrix} 0\\0\end{matrix}\\
\begin{matrix}0&0\\0&0\\0&0\\0&0 \end{matrix} & \begin{matrix} e^{\lambda_1(A_k)} \\ 0 \\ 0\\ 0 \end{matrix} & \begin{matrix} 0 \\ e^{\lambda_2(A_k)} \\ 0\\ 0 \end{matrix} &\begin{matrix} 0 \\ 0 \\ e^{\lambda_1(A)}\\ 0 \end{matrix}& \begin{matrix} 0 \\ 0 \\ 0\\ e^{\lambda_2(A)} \end{matrix}
\end{array}\right),
\end{displaymath}
it is easy to see that it converges to 
\begin{displaymath}
\hat{A}=\left( \begin{array}{ccccc}
A & \begin{matrix} 0\\0\end{matrix} & \begin{matrix} 0\\0\end{matrix}&\begin{matrix} 0\\0\end{matrix}&\begin{matrix} 0\\0\end{matrix}\\
\begin{matrix}
0&0\\0&0\\0&0\\0&0
\end{matrix} & \begin{matrix} e^{a} \\ 0 \\ 0\\ 0 \end{matrix} & \begin{matrix} 0 \\ e^{-a} \\ 0\\ 0 \end{matrix} &\begin{matrix} 0 \\ 0 \\ e^{\lambda_1(A)}\\ 0 \end{matrix}& \begin{matrix} 0 \\ 0 \\ 0\\ e^{\lambda_2(A)} \end{matrix}
\end{array}\right)
\end{displaymath}
and moreover $\lim_{k\to \infty}\lambda_i(\hat{A}_k)=\lambda_i(\hat{A})$ for every $i$ but the Oseledets subspaces of $\hat{A}_k$ does not converge to those of $\hat{A}$ with respect to the measure $\mu$. Indeed, Oseledets subspaces corresponding to $\lambda_1(\hat{A}_k)=\lambda_1(A)$ are one dimensional while the ones corresponding to $\lambda_1(\hat{A})=\lambda_1(A)$ are two dimensional and hence there is no convergence in measure. In particular, as expected, the condition $\lim_{k\to \infty}\lambda_i(A_k)=\lambda_i(A)$ for every $i=1,\ldots ,l$ is weaker than $\lim_{k\to \infty}\gamma_i(A_k)=\gamma_i(A)$ for every $i=1,\ldots ,d$. 
\end{remark}

As a simple consequence of our main theorem we get
\begin{corollary}\label{teoremaA}
  $A$ is a continuity point for the Lyapunov exponents if and only if it is a continuity point for the Oseledets subspaces with respect to the measure $\mu$.
 \end{corollary}

It is worth noticing that the proof presented bellow also works with obvious adjustments if we allow the base dynamics $f$ to vary. More precisely, if we consider a sequence of ergodic $\mu$-measure preserving maps $f_k:M\to M$ converging uniformly to $f:M\to M$ and a sequence $\{A_k\}_k \subset C^0(M)$ converging to $A\in C^0(M)$, then a similar statement to the one of Theorem \ref{teo.principal} also works for Lyapunov exponents and Oseledets subspaces of $(A_k,f_k)$ and $(A,f)$. We write the proof in the case when the base dynamics is fixed just to avoid unnecessary notational complications. 

We also observe that our results can be extended to a continuous-time version. Indeed, let $\phi^t:M\mapsto M$, $t\in \real$, be a continuous flow and $A^t:M\rightarrow M(d, \mathbb{R})$, $t\geq 0$, be such that $A^{t+s}(x)=A^t(\phi^s(x))\circ A^s(x)$. This defines a continuous-time semi-invertible cocycle $(\phi^t,A^t)$. Taking $f=\phi^1$ and $A=A^1$, the cocycle $(f,A)$ is a semi-invertible linear cocycle. Moreover, it has the same Lyapunov exponents and Oseledets decomposition as its continuous version. Thus, our results extends directly to this setting.

\section{Preliminary Results}

This section is devoted to present some preliminary results that are going to be used in the proof of Theorem \ref{teo.principal}. We retain all the notation introduced at the previous section.

\subsection{Semi-projective cocycles}
Let $\Pd$ denote the real $(d-1)$-dimensional projective space, that is, the space of all one-dimensional subspaces of $\real ^d$. Given a continuous map $A: M\to M(d,\mathbb{R})$, we want to define an action on $\Pd$ which is, in some sense, induced by $A$. If $(x,[v])\in M\times \Pd$ is such that $A(x)v\neq 0$ then we have a natural action induced by $A$ on $\Pd$ which is just given by $A(x)\left[v\right]=\left[A(x)v\right]$. The difficulty appears when $A(x)v=0$ for some $v\neq 0$. To bypass this issue, let us consider the closed set given by
\begin{displaymath}
\ker(A)=\lbrace (x,[v])\in M\times \Pd ; \;A(x)v=0 \rbrace .
\end{displaymath}
If $\mu(\pi(\ker(A)))=0$ where $\pi :M\times \Pd \to M$ denotes the canonical projection on the first coordinate, then $A(x)$ is invertible for $\mu$-almost every $x\in M$ and hence it naturally induces a map on $\Pd$ which is defined $\mu$-almost everywhere and is all we need. Otherwise, if $\mu(\pi(\ker(A)))>0$ let us consider the set
\begin{displaymath}
K(A)=\lbrace (x,[v])\in M\times \Pd ; \; A^n(x)v=0\mbox{ for some }n>0 \rbrace.
\end{displaymath}
Observe that $K(A)\cap\lbrace x \rbrace \times \Pd \subset \lbrace x \rbrace \times E^{l,A}_x$ for every \textit{regular} point $x\in M$.

Since $\pi(K(A))$ is an $f$-invariant set and $\mu$ is ergodic it follows that $\mu(\pi(K(A)))=1$. Thus, we can define a mensurable section $\sigma:M\to \Pd$ such that $(x,\sigma(x))\in K(A)$. Moreover, we can do this in a way such that if $x\in \pi(\ker(A))$ then $(x,\sigma(x))\in \ker(A)$. Fix such a section. We now define the \textit{semi-projective cocycle} associated to $A$ and $f$ as being the map $F_{A}:M\times \Pd\to M\times \Pd$ given by 
\begin{displaymath}
F_{A} (x,\left[v\right])=\left\lbrace\begin{array}{c} 
                         (f(x),\left[ A(x)v\right]) \mbox{ if }A(x)v\neq 0\\
                         (f(x),\sigma(f(x)) \mbox{ if }A(x)v= 0.\\
                        \end{array}
                        \right.
\end{displaymath}
This is a measurable function which coincides with the usual projective cocycle outside $\ker(A)$. In particular, it is continuous outside $\ker(A)$. From now on, given a non-zero element $v\in \mathbb{R}^{d}$ we are going to use the same notation to denote its equivalence class in $\Pd$. 

Given a measure $m$ on $M\times \Pd$, observe that if $m(\ker(A))=0$ then ${F_A}_{\ast}m$ does not depend on the way the section $\sigma$ was chosen. Indeed, if $\psi:M\times \Pd\to \real$ is a mensurable function then
\begin{displaymath}
 \int_{M\times \Pd}\psi\circ F_A dm= \int_{M\times \Pd \setminus \ker(A)}\psi\circ F_A dm.
\end{displaymath}

In the sequel, we will be primarily interested in $F_A$-invariant measures on $M\times \Pd$ that \textit{projects} on $\mu$, that is, $\pi _{\ast}m=\mu$ and such that $m(\ker(A))=0$. Our first result states if the cocycle $A$ has two different Lyapunov exponents then any such a measure may be written as a convex combination of measures concentrated on a suitable combination of the Oseledets subspaces. An useful notation that we are going to use through the paper is the following: 
$$E_x^{s_i,A}=E^{i+1,A}_x\oplus \cdots \oplus E^{l,A}_x$$
and
$$E_x^{u_i,A}=E^{1,A}_x\oplus \cdots \oplus E^{i,A}_x$$
which denotes, respectively, the Oseledets slow and fast subspaces of `order i' associated to $A$ and
$$E^{i,A}=\{(x,v)\in M\times \Pd ; \; v\in E^{i,A}_x\}.$$

\begin{proposition}\label{decomposition}
If $\gamma_i(A)>\gamma_{i+1}(A)$ then every $F_A$-invariant measure projecting to $\mu$ and such that $m(\ker(A))=0$ is of the form $m=a m^{u_i}+b m^{s_i}$ for some $a,b\in [0,1]$ such that $a+b=1$, where $m^{\ast}$ is an $F_A$-invariant measure projecting on $\mu$ such that its disintegration $\{m^{\ast}_x\}_{x\in M}$ with respect to $\mu$ satisfies $m^{\ast}_x(E^{\ast}_x)=1$ for $\ast\in\lbrace s_i,u_i \rbrace$. 
\end{proposition}

\begin{proof}
Given $j\in \mathbb{N}$ let us consider the set
 $$B_j=\left\lbrace (x,v)\in M\times \Pd; \;\lvert \sin \measuredangle(v,E^{\ast}_x)\rvert \geq \frac{1}{j}\lvert \sin \measuredangle(E^{u_i}_x,E^{s_i}_x)\rvert\mbox{ for }\ast=s_i,u_i\right\rbrace.$$ 
 
Since $\gamma_i(A)>\gamma_{i+1}(A)$ it follows that for any $(x,v)\in B_j$, the angle between $A^n(x)v$ and $E^{u_i}_{f^n(x)}$ decays exponentially fast when $n$ goes to $+\infty$.
Therefore, since by Oseledets' theorem the angle $\measuredangle(E^{u_i}_x,E^{s_i}_x)$ decays sub exponentially it follows that every $(x,v)\in B_j$ leaves $B_j$. Consequently, by Poincar\'e's recurrence theorem $m(B_j)=0$ for every $j\in \mathbb{N}$. Hence, the measure $m$ is concentrated on $\lbrace (x,E^{u_i}_x); \; x\in M\} \cup \{(x,E^{s_i}_x); \; x\in M\rbrace$. Let $\{m_x\}_{x\in M}$ be a disintegration of $m$ with respect to $\mu$. It follows then by the previous observations that $m_x(E^{s_i}_x)+m_x(E^{u_i}_x)=1$ for $\mu$-almost every $x\in M$. Thus, letting $m^{\ast}_x$ be the normalized restriction of $m_x$ to $E^{\ast}_x$ for $\ast \in \{s_i,u_i\}$ we get that $m_x=a(x)m^{u_i}_x+b(x) m^{s_i}_x$ where $a(x)=m_x(E^{u_i}_x)$ and $b(x)=m_x(E^{s_i}_x)$. To conclude the proof, since our measure $\mu$ is ergodic, it only remains to observe that both $a$ and $b$ are invariant functions and consequently constant functions. This follows easily from the invariance of the Oseledets spaces and the fact that,
 since $m$ is $F_A$-invariant, $m_{f(x)}
= A(x)_{\ast}m_x$ for $\mu$-almost every $x\in M$. Indeed,
\begin{equation*}
\begin{split}
 a(f(x)) & = m_{f(x)}(E^{u_i}_{f(x)})= A(x)_{\ast}m_x(E^{u_i}_{f(x)})\\
  &= m_x(E^{u_i}_{x})=a(x)
	\end{split}
\end{equation*}
as we want.
 \end{proof}
 
 Our next result gives the existence of $F_A$-invariant measures concentrated on Oseledets subspaces. This is going to be used in Section \ref{sec: cont Osel implies cont Lyap}.

 \begin{proposition}\label{lyapunov.measures}
  For every $1\leq j < l$, there exists an $F_A$-invariant measure $m$ projecting to $\mu$ and concentrated on $E^{j,A}=\lbrace (x,v) \in M\times \Pd ; \; v\in E^{j,A}_x\rbrace$. In particular, it satisfies $m(\ker(A))=0$.
 \end{proposition}
 
\begin{proof}
 Let $\mathcal{M}_j$ be the space of all probability measures on $M\times \Pd$ such that $m(E^{j,A})=1$ and $\pi_{\ast}m=\mu$. In particular, $m(\ker(A))=0$ for every $m\in \mathcal{M}_j$. 
 
 Let us consider now the map ${F_A}_{\ast}:\mathcal{M}_j\to \mathcal{M}_j$ given by ${F_A}_{\ast}m$. From the invariance of $E^{j,A}$ and the definition of $\mathcal{M}_j$ it follows that ${F_A}_{\ast}$ is well defined and moreover does not depend on the choice of the section $\sigma$ in the definition of the semi-projective cocycle. Furthermore, it is continuous. Indeed, let $\{m_k\}_k \subset \mathcal{M}_j$ be a sequence converging to $m$ in the weak$^{\ast}$ topology and $\psi:M\times \Pd\to \real$ a continuous map. By Lusin's Theorem, given $\epsilon>0$ there exist a compact set $K\subset M$ such that $\mu(M\setminus K)<\frac{\epsilon}{4\norm{\psi}}$ and $x\to E^{j,A}_x$ is continuous when restricted to $K$. Now, since $\ker(A)\cap E^{j,A}=\emptyset$ and $\psi\circ F_A$ is continuous outside $\ker(A)$, it follows from Tietze extension theorem that there exist a continuous function $\hat{\psi}:M\times \Pd\to \real$ satisfying $\hat{\psi}(p)=\psi\circ F_A(p)$ for every $p\in \lbrace (x,v)\in K \times \Pd ; 
\; v\in E^{j,A}_x\rbrace$ and $\norm{\hat{\psi}}\leq \norm{\psi}$. Then,
  \begin{displaymath}
  \left| \int \psi\circ F_A dm_k-\int \psi \circ F_A dm\right| \leq \left| \int \hat{\psi}dm_k-\int \hat{\psi}dm \right| +\epsilon.
 \end{displaymath}
Consequently, taking $k$ sufficiently large, $\lvert \int \psi\circ F_A dm_k-\int \psi \circ F_A dm\rvert<2\epsilon$ as we claimed.

We observe now that $\mathcal{M}_j$ is a closed subset of the set of all probability measures of $M\times \Pd$. In fact, let $\{m_k\}_k \subset \mathcal{M}_j$ be a sequence converging to $m$. As before, given $\epsilon >0$ there exists a compact set $K\subset M$  such that $\mu(M\setminus K)<\epsilon$ and $x\to E^{j,A}_x$ is continuous when restricted to $K$. Thus, since $E^{j,A}_K:=\lbrace (x,v) \in K\times \Pd ; \; v\in E^{j,A}_x\rbrace$ is a closed subset of $M\times \Pd$, it follows that
\begin{displaymath}
 m(E^{j,A})\geq m(E^{j,A}_K)\geq \limsup_{k\to \infty}m_k(E^{j,A}_K).
\end{displaymath}
Therefore, as $m_k(E^{j,A})=1$ and $\mu(M\setminus K)<\epsilon$ we get that $m_k(E^{j,A}_K)>1-\epsilon$ for every $k$ and consequently, since $\epsilon >0$ was arbitrary and $m$ is a probability measure, $m(E^{j,A})=1$ and $\mathcal{M}_j$ is closed.

To conclude the proof, it only remains to observe that given any $m\in \mathcal{M}_j$, every accumulation point of $\frac{1}{n}\sum_{k=0}^{n-1}{F^k_A}_{\ast}m$ gives rise to an $F_A$-invariant measure concentrated on $E^{j,A}$. This follows easily from the previous observations.
  \end{proof}

\begin{remark}\label{measure.realize.lyapunov}
 Letting $\varphi _{A}: M\times \Pd \to \real$ be the map given by
\begin{displaymath}
\varphi _{A}(x,v)=\log \frac{\parallel A(x)v\parallel }{\parallel v \parallel},
\end{displaymath}
it follows easily from the definition and Birkhoff's ergodic theorem that, for every $F_A$-invariant probability measure $m$ concentrated on $E^{j,A}$ and projecting to $\mu$, 
\begin{equation*}
\lambda_j(A)=\int _{M\times \Pd} \varphi_{A}(x,v)dm.
\end{equation*}
\end{remark}

\subsection{The adjoint cocycle} Given $x\in M$, let $A_{\ast}(x):(\real ^d)^{\ast} \to (\real ^d)^{\ast}$ be the adjoint operator of $A(f^{-1}(x))$ defined by 
\begin{equation}\label{eq: definition adjoint}
(A_{\ast}(x)u)v = u(A(f^{-1}(x))v) \; \mbox{for each} \; u\in (\real ^d)^{\ast} \; \mbox{and} \; v\in \real ^d.
\end{equation}
Fixing some inner product $\langle \; ,\; \rangle$ on $\real ^d$ and identifying the dual space $(\real ^d)^{\ast}$ with $\real ^d$ we get the map $A_{\ast}:M\to M(d,\mathbb{R})$ and equation \eqref{eq: definition adjoint} becomes
\begin{displaymath}
\langle A(f^{-1}(x))u,v\rangle = \langle u,A_{\ast}(x)v\rangle \; \mbox{for every}\; u,v\in \real ^d.
\end{displaymath}
The \textit{adjoint cocycle} of $A$ is then defined as the cocycle generated by the map $A_{\ast}:M\to M(d,\mathbb{R})$ over $f^{-1}:M\to M$.

An useful remark is that the Lyapunov exponents counted with multiplicities of the adjoint cocycle are the same as those of the original cocycle. This follows from the fact that a matrix $B$ and its transpose $B^T$ have the same singular values combined with Kingman's sub-additive theorem. Moreover, Oseledets subspaces of the adjoint cocycle are strongly related with the ones of the original cocycle. More precisely,

\begin{lemma}\label{adjoint.spaces}
$E^{s_i,A}_x=(E^{u_i,A_{\ast}}_x)^{\bot}$ where the right-hand side denotes the orthogonal complement of the space $E^{u_i,A_{\ast}}_x$. 
\end{lemma}
\begin{proof}
By contradiction, suppose there exist $v\in E^{u_i,A}_x$ and $u\in E^{s_i,A_{\ast}}_x$ such that $\langle v,u\rangle\neq 0$. We may assume $i<l$ otherwise the lemma trivially holds. In this case, for each $n\in \mathbb{N}$ the map $A^n(f^{-n}(x)):E^{u_i,A}_ {f^{-n}(x)} \to E^{u_i,A}_x$ is surjective and thus we may find unitary vectors $v_n\in E^{u_i,A}_ {f^{-n}(x)}$ such that $A^n(f^{-n}(x))v_n$ are multiples of $v$. By definition,
\begin{displaymath}
\begin{split}
\langle A^n(f^{-n}(x))v_n,u\rangle &= \langle v_n, (A^n(f^{-n}(x)))^{\ast}u\rangle \\
&=\langle v_n, A_{\ast}^n(x)u\rangle.
\end{split}
\end{displaymath}
Now, since $\langle v_n, A_{\ast}^n(x)u\rangle$ grows at an exponential rate smaller than $\lambda _i(A)$ while $\langle A^n(f^{-n}(x))v_n,u\rangle$ grows at an exponential rate at least $\lambda _i(A)$ we get a contradiction. Therefore, $E^{u_i,A}_x \subset (E^{s_i,A_{\ast}}_x)^{\bot}$. Now, since they have the same dimension we get $E^{u_i,A}_x = (E^{s_i,A_{\ast}}_x)^{\bot}$. Finally, observing that $(A_{\ast})_{ \ast}=A$ the lemma follows.

\end{proof}

\section{Continuity of Lyapunov exponents implies continuity of Oseledets subspaces}

At this section we are going to prove that continuity of Lyapunov exponents implies continuity of Oseledets subspaces. Thus, let $\{A_k\}_k \subset C^0(M)$ be a sequence converging to $A\in C^0(M)$ and suppose $\lim_{k\to \infty}\gamma_i(A_k)= \gamma _i(A)$ for every $1\leq i \leq d$. We start with an auxiliary lemma.

\begin{lemma}\label{ker.zero}
Let $m_k$ be a sequence of $F_{A_k}$-invariant measures concentrated on $E^{1,A_k}$ and suppose they converge to a measure $m$. Then $m(\ker(A))=0$ and moreover $m$ is an $F_A$-invariant measure.
\end{lemma}

\begin{proof}
We start proving that $m(\ker(A))=0$. Suppose by contradiction that $m(\ker(A))=2 c>0$. For each $\delta >0$ let us consider 
\begin{displaymath}
K_{\delta}=\left\{ (x,v)\in M\times \Pd ; \; \norm{A(x)\frac{v}{\norm{v}}}<\delta \right\}.
\end{displaymath}
These are open sets such that $\ker(A)=\cap_{\delta >0} K_{\delta}$ and $m(K_{\delta})\geq m(\ker(A))>c>0$. 

Fix $b\in \real$ such that
$$b< \gamma_1(A)-\sup_{k,x,\norm{v}=1} \log \norm{A_k(x)v}$$ 
and let $\delta>0$ be such that $\log y<\frac{b}{c}$ for every $y<2\delta$. Then, for every $k$ sufficiently large $m_k(K_{\delta})>c>0$ and $\norm{A_k(x)\frac{v}{\norm{v}}}<2\delta$ for every $(x,v)\in K_{\delta}$ and consequently

\begin{displaymath}
\gamma_1(A_k)=\int \varphi_{A_k}d m_k<b+\sup_{k,x,\norm{v}=1} \log \norm{A_k(x)v}
\end{displaymath}
contradicting the choice of $b$. Thus, $m(\ker(A))=0$ as we want.

To prove that $m$ is $F_A$-invariant one only has to show that, given a continuous map $\psi:M\times \Pd\to \real$, 
\begin{equation} \label{eq: FA invariance}
\lim_{k\to \infty} \int \psi\circ F_{A_k} dm_k=\int \psi\circ F_{A} dm.
\end{equation}
Indeed, if \eqref{eq: FA invariance} is true then, since $m_k$ is $F_{A_k}$-invariant, 
\begin{displaymath}
\int \psi \circ F_{A} dm=\lim_{k\to \infty} \int \psi\circ F_{A_k} dm_k=\lim _{k\to \infty} \int \psi dm_k=\int \psi dm.
\end{displaymath}

In order to prove $\eqref{eq: FA invariance}$ we start noticing that
\begin{displaymath}
\begin{split}
\left| \int \psi\circ F_{A_k} dm_k-\int \psi\circ F_{A} dm\right| &\leq  \int \vert \psi\circ F_{A_k}-\psi\circ F_{A} \vert dm_k \\
&+\left| \int \psi\circ F_{A} dm_k-\int \psi\circ F_{A} dm \right|.
\end{split}
\end{displaymath} 

Now observing that, for every $k$ sufficiently large, $\norm{ A_k(x)v/\norm{v}}>\frac{\delta}{2}$ if $(x,v)\in K_{\delta} ^c$ and recalling the definition of semi-projective cocycle it follows that $\psi\circ F_{A_k}$ converges uniformly to $\psi \circ F_{A}$ outside $K_{\delta}$. Given $\varepsilon >0$ let $\delta >0$ be such that $m(\overline{K_{\delta}})<\frac{\epsilon}{2\norm{\psi}}$. Then, taking $k$ sufficiently large such that $\vert \psi\circ F_{A_k} - \psi\circ F_{A}\vert <\epsilon$ outside $K_{\delta}$ and $m_k(\overline{K_{\delta}})<\frac{\epsilon}{2\norm{\psi}}$ we get
\begin{displaymath}
 \int \vert \psi\circ F_{A_k}-\psi\circ F_{A} \vert dm_k< 2\epsilon.
\end{displaymath} 

To bound $\left| \int \psi\circ F_{A} dm_k-\int \psi\circ F_{A} dm \right|$, let $\hat{\psi}:M\times \Pd\to \real$ be a continuous function which is equal to $\psi\circ F_{A}$ outside $K_{\delta}$ and $\norm{\hat{\psi}}\leq \norm{\psi}$. Note that the existence of such a map is guaranteed once again by Tietze extension theorem. Then, 
\begin{displaymath}
\left| \int \psi\circ F_{A} dm_k-\int \psi\circ F_{A} dm \right| \leq \left| \int \hat{\psi}dm_k-\int \hat{\psi}dm\right| +2\epsilon.
\end{displaymath}
Now, taking $k$ sufficiently large such that $ \left| \int \hat{\psi}dm_k-\int \hat{\psi}dm\right|<\epsilon$ it follows that
\begin{displaymath}
\left| \int \psi\circ F_{A_k} dm_k-\int \psi\circ F_{A} dm\right| < 5\epsilon
\end{displaymath} 
proving \eqref{eq: FA invariance} and consequently the lemma.

\end{proof}

\begin{remark}\label{remark.ker.zero}
Observe that in the proof of the previous lemma we didn't use the full strength of the requirement $\lim_{k\to \infty}\gamma_i(A_k)= \gamma _i(A)$ for every $1\leq i\leq d$. Indeed, it is enough that $\lim_{k\to \infty}\int \varphi_{A_k} dm_k>-\infty$. This is going to be used in Section \ref{sec: cont Osel implies cont Lyap}.
\end{remark}

\subsection{Continuity of the fastest Oseledets subspace}
Our next proposition deals with the case when $d_1(A)=1$. That is, the case when the dimension of the Oseledets subspace associated with $\lambda _1(A)$ is $1$.

\begin{proposition}\label{mayor.exponente}
If $A$ is such that $\gamma_1(A)>\gamma_2(A)$ then $E^{1,A_k}_x$ converges to $E^{1,A}_x$ with respect to the measure $\mu$. More precisely, for  every $\delta >0$
\begin{displaymath}
\mu (\{x\in M; \; \measuredangle(E^{1,A_k}_x,E^{1,A}_x)<\delta\})\xrightarrow{k\to \infty}1.
\end{displaymath}
\end{proposition}
 
\begin{proof}
We start observing that, since $\gamma _j(A_k)\xrightarrow{k\to \infty} \gamma _j(A)$ for every $1\leq j\leq d$ and $\gamma_1(A)>\gamma_2(A)$, for every $k$ sufficiently large $\gamma _1(A_k)>\gamma _2(A_k)$ and thus $E^{1,A_k}_x$ is also one-dimensional. Let us assume without loss of generality that this is indeed the case for every $k\in \mathbb{N}$.

For each $k\in \mathbb{N}$, let us consider the measure
\begin{displaymath}
m_k=\int _{M} \delta _{(x,E^{1,A_k}_{x})} d\mu(x)
\end{displaymath}
and let $m^u$ be the measure given by
\begin{displaymath}
m^u=\int _{M} \delta _{(x,E^{1,A}_{x})} d\mu(x).
\end{displaymath}
Observe that these are, respectively, $F_{A_k}$ and $F_A$-invariant measures on $M\times \Pd$ concentrated on $E^{1,A_k}$ and $E^{1,A}$ and projecting to $\mu$. Consequently, it follows from Remark \ref{measure.realize.lyapunov} that
\begin{equation}\label{eq: measure realizing Lyap exp Ak}
\gamma _1(A_k)=\int _{M\times \Pd} \varphi _{A_k}(x,v)dm_k
\end{equation}
and
\begin{equation}\label{eq: measure realizing Lyap exp A}
\gamma _1(A)=\int _{M\times \Pd} \varphi _{A}(x,v)dm^u.
\end{equation}

We claim now that $m_k$ converges to $m^u$ in the weak$^*$ topology. Indeed, let $\{m_{k_i}\}_{i\in \mathbb{N}}$ be a convergent subsequence of $\{m_{k}\}_{k\in \mathbb{N}}$ and suppose that it converges to $m$. Since $M \times \Pd$ is a compact space it suffices to prove that $m=m^u$. Observing that, for each $i\in \mathbb{N}$ the measure $m_{k_i}$ is an $F_{A_{k_i}}$-invariant measure projecting to $\mu$, it follows from Lemma \ref{ker.zero} that $m$ is an $F_A$-invariant measure projecting to $\mu$ and moreover $m(\ker (A))=0$. Furthermore, since
\begin{displaymath}
\gamma _1(A_{k_i})\xrightarrow{i\to +\infty}\gamma _1(A)
\end{displaymath}
and 
\begin{displaymath}
\int _{M\times \Pd} \varphi _{A_{k_i}}(x,v)dm_{k_i}\xrightarrow{i\to +\infty} \int _{M\times \Pd} \varphi _{A}(x,v)dm
\end{displaymath}
it follows from \eqref{eq: measure realizing Lyap exp Ak} that 
\begin{displaymath}
\gamma _1(A)=\int _{M\times \Pd} \varphi _{A}(x,v)dm.
\end{displaymath}
Thus, from Proposition~\ref{decomposition} we get that $m=m^u$ as claimed. In fact, otherwise we would have $m=am^{u_1}+bm^{s_1}$ where $a,b\in (0,1)$ are such that $a+b=1$ and $m^{s_1}$ is an $F_A$-invariant measure concentrated on $\{(x,E^{s_1}_x); \; x\in M\}$. 
Therefore,
\begin{displaymath}
\begin{split}
\gamma _1(A)&=\int _{M\times \Pd} \varphi _{A}(x,v)dm \\
&=a\int _{M\times \Pd} \varphi _{A}(x,v)dm^{u_1} +b\int _{M\times \Pd} \varphi _{A}(x,v)dm^{s_1} \\
&\leq a \gamma _1(A)+b\gamma _2(A) < \gamma _1(A).
\end{split}
\end{displaymath}

Let us consider now the measurable map $\psi:M \to \Pd$ given by 
\begin{displaymath}
\psi(x)=E^{1, A}_{x}. 
\end{displaymath} 
Note that its graph has full $m^u$-measure. By Lusin's Theorem, given $\varepsilon >0$ there exists a compact set $K\subset M$ such that the restriction $\psi _K$ of $\psi$ to $K$ is continuous and $\mu (K)>1-\varepsilon$. Now, given $\delta >0$, let $V\subset M\times \Pd$ be an open neighborhood of the graph of $\psi_K$ such that
\begin{displaymath}
V\cap (K\times \Pd)\subset V_{\delta}
\end{displaymath}
where
\begin{displaymath}
V_{\delta}:=\lbrace (x,v)\in K\times \Pd ; \; \measuredangle (v, \psi (x))<\delta \rbrace .
\end{displaymath}
By the choice of the measures $m_k$,
\begin{equation}\label{eq: auxiliary eq 1 prop 2}
m_k(V_{\delta})= \mu (\lbrace x\in K; \; \measuredangle (E^{1,A_k}_{x}, E^{1,A}_{x})<\delta \rbrace ).
\end{equation}
Now, as $m_k\xrightarrow{k\rightarrow \infty} m^u$ it follows that $\liminf m_k(V)\geq m^u(V)> 1-\varepsilon$. On the other hand, as $m_k(K\times \Pd)=\mu(K)> 1-\varepsilon$ for every $k\in \mathbb{N}$, it follows that 
\begin{equation}\label{eq: auxiliary eq 2 prop 2}
m_k(V_{\delta})\geq m_k(V \cap (K\times \Pd))\geq 1-2\varepsilon
\end{equation}
for every $k$ large enough. Thus, combining \eqref{eq: auxiliary eq 1 prop 2} and \eqref{eq: auxiliary eq 2 prop 2}, we get that $\mu(\lbrace x\in M; \; \measuredangle (E^{1,A_k}_{x}, E^{1,A}_{x})<\delta \rbrace)\geq 1-2\varepsilon$ for every $k$ large enough completing the proof of the proposition.

\end{proof}

\subsection{Continuity of the Oseledets fast subspace of order $i$}
We now prove that the Oseledets fast subspace of order $i$ of $A_k$ converges to the respective Oseledets subspace of $A$. The idea is to consider the cocycle induced by $A$ on a suitable exterior power and then deduce the general case from the previous one. 

\begin{proposition}\label{prop: continuity i fastest Osel}
For every $1\leq i\leq l$ and $\delta >0$ we have that
\begin{displaymath}
\mu (\{x\in M; \; \measuredangle(E^{u_i,A_k}_x,E^{u_i,A}_x)<\delta\})\xrightarrow{k\to \infty}1.
\end{displaymath}
\end{proposition}

Before we proceed to the proof let us just recall some facts about Grassmannian manifolds and exterior powers that we are going to use in the sequel. For a more detailed explanation we just refer to \cite{LLE}.

For every $1\leq j\leq d$ we denote by $\exteriorj$ the $j$th \textit{exterior power} of $\real ^d$ which is the space of alternate $j$-linear forms on the dual $(\real ^d)^{\ast}$. If $\wedge$ denotes the exterior product of vectors of $\real ^d$ then a basis for $\exteriorj$ is given by $\{e_{i_1}\wedge \ldots \wedge e_{i_j}; \; 1\leq i_1<\ldots <i_j\leq l\}$ whenever $\{e_i\}_{i=1}^{d}$ is a basis for $\real ^d$. We may also consider the exterior product $V\wedge W$ of subspaces $V$ and $W$ of $\real ^d$. This is defined as the exterior product of the elements of any basis of $V$ with the elements of any basis of $W$. Any linear map $L\in \mbox{M}(d,\real)$ induces a linear map $\Lambda ^jL:\exteriorj \to \exteriorj$ by 
\begin{displaymath}
\Lambda ^jL (\omega): \phi _1\wedge \ldots \wedge \phi _j\to \phi _1\circ L\wedge\ldots \wedge\phi _j\circ L. 
\end{displaymath}
Hence, a linear cocyle generated by $B:M\to M(d,\real )$ over $f$ induces a linear cocycle over $f$ on the $j$th exterior power which is generated by the map $x\to \Lambda ^jB(x)$. Moreover, if $B$ satisfies the integrability condition so does $\Lambda ^jB$ and its Lyapunov exponents are given by 
\begin{equation} \label{Lyapunov exterior power}
\{\gamma _{i_1}(B) +\ldots +\gamma _{i_j}(B); \; 1\leq i_1<\ldots < i_j\leq l\}.
\end{equation}
Furthermore, Oseledets subspaces of $\Lambda ^jB$ are strongly related with the ones of $B$. In particular, for every $1\leq i\leq l$ the Osleledets subspace of $\Lambda ^{d_1(B)+\ldots +d_i(B)}B$ at the point $x\in M$ associated to $\gamma _{1}(B)+\gamma _2(B) +\ldots +\gamma _{d_1(B)+\ldots +d_i(B)}(B)$ is given by
\begin{equation}\label{Oseledets exterior power}
E^{1,B}_x\wedge \ldots \wedge E^{i,B}_x.
\end{equation}
This is all we are going to use about the Oseledets subspaces of induced cocycle.

Let $\grassj$ denote \textit{Grassmannian manifold} of $j$-dimensional subspaces of $\real ^d$. The map $\psi :\grassj \to \P (\exteriorj )$ which assigns to each subspace $E\in \grassj$ the projective point $[v]\in \P (\exteriorj )$, where $v=v_1\wedge \ldots \wedge v_j$ and $\{v_1,\ldots ,v_j\}$ is any basis for $E$, is an embedding known as the \textit{Pl\"ucker embedding}. Therefore, if $\rho (.,.)$ is a distance on $\P (\exteriorj )$ we may push it back to $\grassj$ via $\psi$. More precisely, the map $\d_{\exteriorj}:\grassj\times \grassj\to \real$ given by
\begin{displaymath}
\d_{\exteriorj} (E_1,E_2)=\rho (\psi (E_1),\psi (E_2))
\end{displaymath}
is a distance on $\grassj$ and moreover, if $\rho$ is a distance given by an inner product in the linear space $\exteriorj$ then $\d_{\exteriorj}$ is equivalent to the distance defined in \eqref{def:distancia}.

\begin{proof}[Proof of Proposition \ref{prop: continuity i fastest Osel}]
Observe that if $i=l$ then there is noting to do since $E^{u_l,A_k}_x=\real ^d=E^{u_l,A}_x$ for every $k$ sufficiently large. So, from now on let us assume $i<l$. 

Consider $r=d_1(A)+\ldots +d_i(A)$ and let $\Lambda ^rA$ and $\Lambda ^rA_k$ be the cocycles over $f$ induced by $A$ and $A_k$, respectively, on the $r$th exterior power. Since we are assuming $i<l$ it follows from \eqref{Lyapunov exterior power} that $\gamma _1 (\Lambda ^rA)>\gamma _2 (\Lambda ^rA)$. Thus, from Proposition \ref{mayor.exponente} we get that, for every $\delta' >0$, 
\begin{displaymath}
\mu (\{x\in M; \; \measuredangle(E^{1,\Lambda ^rA_k}_x,E^{1,\Lambda ^rA}_x)<\delta'\})\xrightarrow{k\to \infty}1
\end{displaymath}
which from \eqref{Oseledets exterior power} is equivalent to
\begin{displaymath}
\mu (\{x\in M; \; \measuredangle(E^{1,A_k}_x\wedge \ldots \wedge E^{i,A_k}_x,E^{1,A}_x\wedge \ldots \wedge E^{i,A}_x)<\delta'\})\xrightarrow{k\to \infty}1.
\end{displaymath}
Consequently, from the definition of $\d_{\Lambda ^r(\real ^d)}$ it follows that
\begin{displaymath}
\mu (\{x\in M; \; \d_{\Lambda ^r(\real ^d)}(E^{1,A_k}_x \oplus \ldots \oplus E^{i,A_k}_x,E^{1,A}_x\oplus \ldots \oplus E^{i,A}_x)<\delta ' \})\xrightarrow{k\to \infty}1.
\end{displaymath}
Now, using the fact that the distances $\d_{\Lambda ^r(\real ^d)}$ and $\d$ are equivalent it follows that for every $\delta >0$, 
\begin{displaymath}
\mu (\{x\in M; \; \measuredangle(E^{1,A_k}_x \oplus \ldots \oplus E^{i,A_k}_x,E^{1,A}_x\oplus \ldots \oplus E^{i,A}_x)<\delta\})\xrightarrow{k\to \infty}1
\end{displaymath}
as we want.
\end{proof}

As a simple consequence of the previous proposition applied to adjoint cocycles $A_{k _{\ast}}$ and $A_{\ast}$ combined with Lemma \ref{adjoint.spaces} we get that
\begin{corollary}\label{cor: weakest Oseledets}
For every $0\leq i\leq l-1$ and $\delta >0$ we have that
\begin{displaymath}
\mu (\{x\in M; \; \measuredangle(E^{s_i,A_k}_x,E^{s_i,A}_x)<\delta\})\xrightarrow{k\to \infty}1.
\end{displaymath}
\end{corollary}

\subsection{Proof of the direct implication of Theorem \ref{teo.principal}}

The \textit{cone} of radius $\alpha >0$ around a subspace $V$ of $\real ^d$ is defined as  
\begin{displaymath}
C_{\alpha}(V)=\left\{ w_1+w_2 \in V \ \oplus V^{\perp}; \; \norm{w_2}<\alpha \norm{w_1}\right\}.
\end{displaymath}
Observe that this is equivalent to 
\begin{displaymath}
C_{\alpha}(V)=\left\{ w\in \real^d ; \; \d\left(\frac{w}{\norm{w}},V\right)<\alpha\right\}
\end{displaymath}
where $\d$ is the distance defined in \eqref{def:distancia}.

In order to prove the direct implication of our main theorem we are going to need the following auxiliary result.

\begin{lemma}\label{intersection}
Given $1\leq i\leq l$, $\epsilon>0$ and $\delta>0$ there exist a subset $K=K(\epsilon)\subset M$ with $\mu(K)>1-\epsilon$ and $\delta'=\delta'(\epsilon,\delta)>0$, such that for every $x\in K$, 
$$ C_{\delta'}(E^{u_i,A}_x)\cap C_{\delta'}(E^{s_{i-1},A}_x)\subset C_{\delta}(E^{i,A}_x).$$
\end{lemma}
\begin{proof}
For every regular point $x\in M$ we can define an inner product $\ip{ , }_x$ on $\real ^d$ such that $\{E^{i,A}_x\}_{i=1}^{l}$ are mutually orthogonal. Moreover, this family of inner products may be chosen to be measurable. Let $K\subset M$ be a compact subset of $M$ with $\mu(K)>1-\epsilon$ and such that $\ip{ , }_x$ is continuous when restricted to $K$. Then, there exists $C>1$ such that $\frac{1}{C}\norm{v}\leq\norm{v}_x\leq C\norm{v}$. Take $\delta':=\frac{\delta}{4C^2}>0$.

Given $v\in C_{\delta'}(E^{u_i,A}_x)\cap C_{\delta'}(E^{s_{i-1},A}_x)$, for every $x\in K$ we can write $v=v_{i}+v_{u_{i-1}}+v_{u_{i}}^{\perp}$ where
\begin{displaymath}
v_{i}=\mbox{Proj}_{E^{i,A}_x}(v)\mbox{, }\; v_{u_{i-1}}=\mbox{Proj}_{E^{u_{i-1},A}_x}(v)\; \mbox{ and } \; v_{u_{i}}^{\perp} =\mbox{Proj} _{({E^{u_i,A}_x})^{\perp}}(v).
\end{displaymath}

 Analogously $v=v_{i}+v_{s_i}+v_{s_{i-1}}^{\perp}$. From the definition of cone we get that $\norm{v_{u_i}^{\perp}}<\delta '$ and $\norm{v_{s_{i-1}}^{\perp}}<\delta'$ and consequently, 
\begin{equation*}
\norm{v_{s_i}-v_{u_{i-1}}}<2\delta'.
\end{equation*}
Now, from the definition of $v_{s_i}$ and $v_{u_{i-1}}$ and the choice of $C$ it follows that
\begin{equation*}
\norm{v_s}_x<2C\delta'.
\end{equation*}
Consequently,
\begin{equation*}
\norm{v_s}<2C^2\delta'<\frac{\delta}{2}
\end{equation*}
and thus, if $v=v_{i}+v_{i}^{\perp}$ then 
\begin{equation*}
\norm{v_{i}^{\perp}}\leq \norm{v_{s_i}+v_{s_{i-1}}^{\perp}} <\delta
\end{equation*}
which implies that $v\in C_{\delta}(E^{i,A}_x)$ as we want.
\end{proof}

Given $\varepsilon >0$, let $K\subset M$ and $\delta '>0$ be given by the previous lemma. Proposition \ref{prop: continuity i fastest Osel} and Corollary \ref{cor: weakest Oseledets} gives us that for every $1\leq i\leq l$ and $k$ sufficiently large the sets
 $$A^{u_i}=\lbrace x\in M; \; \measuredangle(E^{u_i,A_k}_x,E^{u_i,A}_x)\geq \delta'\rbrace$$
 and
 $$A^{s_{i-1}}=\lbrace x\in M; \; \measuredangle(E^{s_{i-1},A_k}_x,E^{s_{i-1},A}_x)\geq \delta'\rbrace$$
are such that $\mu(A^{u_i})<\epsilon$ and $\mu(A^{s_{i-1}})<\epsilon$. Now, observing that, for $x\notin A^{u_i}\cup A^{s_{i-1}}$ and $k$ sufficiently large,
$$E^{i,A_k}_x=E^{u_i,A_k}_x\cap E^{s_{i-1},A_k}_x\subset C_{\delta'}(E^{u_i,A}_x)\cap C_{\delta'}(E^{s_{i-1},A}_x)$$
 it follows from Lemma \ref{intersection} that, for every $x\in K\setminus(A^{u_i}\cup A^{s_{i-1}})$ and $k$ sufficiently large, $E^{i,A_k}_x \subset C_{\delta}(E^{i,A}_x)$. Consequently, $\mu (\{x\in M; \; \measuredangle(E^{i,A_k}_x,E^{i,A}_x)<\delta'\}) \geq 1-3\varepsilon$ for every $k$ sufficiently large as we want.

\section{Continuity of Oseledets subspaces implies continuity of Lyapunov exponents} \label{sec: cont Osel implies cont Lyap}

This section is devoted to prove the reverse implication of Theorem \ref{teo.principal}. So, let $\{A_k\}_k \subset C^0(M)$ be a sequence converging to $A\in C^0(M)$ and suppose that for every $k$ sufficiently large there exists a direct sum decomposition $\mathbb{R}^d=F^{1,A_k}_{x}\oplus \ldots \oplus F^{l,A_k}_{x}$ into vector subspaces such that 
\begin{itemize}
\item[i)] $F^{i,A_k}_{x}=E^{j,A_k}_{x}\oplus E^{j+1,A_k}_{x}\oplus \ldots \oplus E^{j+t,A_k}_{x}$ for some $j\in \{1,\ldots ,l_k\}$ and $t\geq 0$;
\item[ii)] $\mbox{dim}(F^{i,A_k}_{x})= \mbox{dim}(E^{i,A}_{x})$ for every $i=1,\ldots ,l$
\end{itemize}
and moreover that
\begin{itemize}
\item[iii)] for every $\delta >0$ and $1\leq i\leq l$ we have 
$$\mu\left(\{ x\in M; \measuredangle(F^{i,A_k}_x,E^{i,A}_x )>\delta \} \right) \xrightarrow{k\to \infty} 0.$$
\end{itemize}

Given $1\leq i < l$, we start proving that if  
$$\mu\left(\lbrace x\in M\mbox{, }\measuredangle(F^{i,A_k}_x,E^{i,A}_x)>\delta\rbrace \right) \to 0$$ for every $\delta>0$ then $\gamma _j(A_k) \to \gamma _j(A)$ for every $d_0(A)+d_1(A)+\ldots + d_{i-1}(A)< j \leq d_1(A)+\ldots + d_{i}(A)$ where $d_0(A)=0$.

For each $k\in \mathbb{N}$, let $m_k$ be an $F_{A_k}$-invariant measure supported on $\lbrace (x,v)\in M\times \Pd ; \; v\in F^{i,A_k}_x\rbrace $ which projects to $\mu$ and such that
\begin{equation}\label{eq: def of m_k last sec}
\gamma _j(A_k)=\int \varphi_{A_k}(x,v)d m_k.
\end{equation}
The existence of such a measure is guaranteed by Proposition \ref{lyapunov.measures} and Remark \ref{measure.realize.lyapunov}. Passing to a subsequence we may assume that $m_k$ converges in the weak$^{\ast}$ topology to some measure $m$. From Lemma \ref{ker.zero} it follows that $m$ is an $F_A$-invariant measure projecting to $\mu$ and moreover that $m(\ker (A))=0$. To conclude the proof it suffices to observe that $m$ is supported on $\lbrace (x,v)\in M\times \Pd ; \; v\in E^{i,A}_x\rbrace$. Indeed, if that is the case then invoking Remark \ref{measure.realize.lyapunov} we get
$$\lim_{k\to \infty}\gamma_j(A_k)=\lim _{k\to \infty} \int \varphi_{A_k} dm_k=\int \varphi_{A} dm=\gamma_j(A)$$ 
for every $d_0(A)+d_1(A)+\ldots + d_{i-1}(A)< j \leq d_1(A)+\ldots + d_{i}(A)$ as we want.

Given $\varepsilon>0$, let $K\subset M$ be a compact set with $\mu(K)>1-\frac{\varepsilon}{2}$ and such that $E_x^{i,A}$ is continuous when restricted to $K$. For each $\delta >0$ let us consider 
\begin{displaymath}
G_{\delta}=\lbrace (x,v)\in K\times \Pd ;\; \measuredangle(v,E^{i,A}_x)\leq \delta\rbrace.
\end{displaymath}
This is a closed set and thus, by the weak$^{\ast}$ convergence of the sequence $\{m_k\}_k$, 
\begin{equation}\label{measure of G delta}
m(G_{\delta}) \geq \limsup_{k\to \infty} m_k(G_{\delta}).
\end{equation}
Since $m_k$ projects to $\mu$ it follows by Rokhlin's disintegration theorem that $m_k$ can be written as $m_k=\int m^k_{x}d \mu(x)$ where $\{m_x^k\}_{x\in M}$ are measures on $\Pd$. Moreover, from the choice of $m_k$ it follows that $m_x^k(F^{i,A_k}_x)=1$ for $\mu$-almost every $x\in M$. Consequently,
\begin{equation}\label{G delta set}
m_k(G_{\delta}) =\int m_x^k(G_{\delta})d\mu \geq 1-\mu\left(K^c\cup \lbrace x\in M; \;\measuredangle(F^{i,A_k}_x,E^{i,A}_x)>\delta\rbrace \right).
\end{equation}

Now, let $k_{\delta}\in \mathbb{N}$ be such that $\mu\left(\lbrace x\in M; \;\measuredangle(F^{i,A_k}_x,E^{i,A}_x)>\delta\rbrace \right)<\frac{\varepsilon}{2}$ for every $k\geq k_{\delta}$. Thus, invoking \eqref{G delta set} we get that $m_k(G_{\delta})\geq 1-\varepsilon$ for every $\delta >0$ as far as $k\geq k_{\delta}$. Hence, it follows from \eqref{measure of G delta} that $m(G_{\delta})\geq 1-\varepsilon$ for every $\delta >0$. Consequently,
\begin{displaymath}
\begin{split}
m(\lbrace (x,v)\in M\times \Pd ;\; v\in E^{i,A}_x\rbrace)&\geq m(\lbrace (x,v)\in K\times \Pd; \;v\in E^{i,A}_x\rbrace)\\
&\geq \lim _{\delta \to 0}m(G_{\delta})\geq 1-\varepsilon
\end{split}
\end{displaymath}
Since $\varepsilon >0$ is arbitrary we conclude that  $m(\lbrace (x,v)\in M\times \Pd ;\; v\in E^{i,A}_x\rbrace)=1$ as claimed.

It remains to consider the case when $i=l$. If $\lambda _l(A)>-\infty$, then the previous argument also works for this case. Otherwise, if $\lambda _l(A)=-\infty$ it suffices to prove that $\gamma_j(A_k)\to -\infty$ for $j=d_1(A)+\ldots + d_{l-1}(A)+1$. Suppose that is not the case, that is, $\limsup_{k\to\infty}\gamma_j(A_k)>-\infty$. Passing to a subsequence, if necessary, we may assume that $\lim_{k\to\infty} \gamma_j(A_k)=a>-\infty$ and moreover that the sequence of measures $\{m_k \}_k$ given as in \eqref{eq: def of m_k last sec} converge to some measure $m$. It follows then from Remark \ref{remark.ker.zero} that $m$ is an $F_A$-invariant measure and $m(\ker(A))=0$. Proceeding as we did in the previous case we conclude that $m(E^{l,A})=1$ and 
$$-\infty > a=\lim_{k\to \infty}\gamma_j(A_k)=\lim _{k\to \infty} \int \varphi_{A_k} dm_k=\int \varphi_{A} dm.$$
On the other hand, Birkhoff's ergodic theorem implies that $\int \varphi_{A} dm=-\infty$ which gives us a contradiction. Therefore, $\gamma_j(A_k)\to -\infty$ for $j=d_1(A)+\ldots + d_{l-1}(A)+1$ and hence $\gamma_j(A_k)\to -\infty$ for every $j \in \{ d_1(A)+\ldots + d_{l-1}(A)+1,\ldots , d \}$ completing the proof of Theorem \ref{teo.principal}.


\medskip{\bf Acknowledgements.} We thank to Professor Marcelo Viana for useful discussions regarding the content of this work. We also thank to the anonymous referee for suggestions that helped to improve this presentation. M. Poletti was partially supported by CAPES and CNPq-Brazil.


\bibliography{bibliography}

\end{document}